\documentclass{amsart}
\usepackage{standalone}
\usepackage{amsmath, amsthm, amssymb}
\usepackage{cleveref}
\usepackage{bbm}
\usepackage[margin=1in]{geometry}
\usepackage{tikz}
\usepackage{circuitikz}
\usetikzlibrary{calc,decorations.pathreplacing,angles,quotes}

\newcommand{\NN}{\mathbb{N}}
\newcommand{\RR}{\mathbb{R}}
\newcommand{\SL}{\operatorname{SL}}
\newcommand{\vol}{\operatorname{vol}}

\title{On The Limiting Distribution of Free Path Lengths for Flat Surfaces with Circular Obstacles}
\author{Diaaeldin Taha}
\address{Diaaeldin Taha, Ruprecht-Karls-Universit\"{a}t Heidelberg, Mathematisches Institut, Im Neuenheimer Feld 205, 69120 Heidelberg, Germany}
\email{dtaha@mathi.uni-heidelberg.de}

\newtheorem{proposition}{Proposition}
\newtheorem{lemma}{Lemma}

\newtheorem{theorem}{Theorem}

\begin{document}

\begin{abstract}
In this note, we prove the existence of a limiting distribution of the free path lengths on flat surfaces with circular obstacles as the radius of the obstacles goes to zero. Moreover, we relate this distribution to the distribution of the heights of zippered rectangle decompositions of flat surfaces.
\end{abstract}

\maketitle

\section{Introduction}

The Sinai billiard table is a billiard system on the two-dimensional torus with one or more circular regions (obstacles) removed. The model was originally introduced by Lorentz \cite{lorentz1905mouvement} in 1905 to model the motion of electrons in metals. The system has since been extensively analyzed from the point of view of dynamical systems \cite{bunimovich2000billiards, chernov2000entropy, dahlqvist1997lyapunov, friedman1984universal, gallavotti1975lectures, golse2006statistics, sinai1970dynamical}. One quantity of interest is the free path length between collisions of a free moving particle with the obstacles. The distribution of the free path lengths in the small obstacle limit was first observed by Dahlqvist \cite{dahlqvist1997lyapunov} in 1997, and established rigorously by Boca and Zaharescu \cite{boca2007distribution} in 2007 using analytic number theoretic techniques applied to continued fractions and Farey sequences. The distribution was also independently calculated in 2008 by Caglioti and Golse \cite{caglioti2008boltzmann} using continued fractions, and by Marklof and Str\"{o}mbergsson \cite{marklof2010distribution} using homogeneous dynamics. Remarkably, Marklof-Str\"{o}mbergsson's approach works in any dimension and extends to aperiodic quasicrystalline point configurations.

One fascinating fact about the free path distribution of the Sinai billiard table is that it agrees with the limiting average distribution of the gaps of circle rotations \cite{bleher1991energy, mazel1992limiting,  greenman1996generic, polanco2016continuous}. This connection is demystified when one observes that circle rotation gaps can be interpreted as zippered rectangle heights of tori \cite{taha2021threegap}, and that these heights approximate the free path lengths. In fact, the latter observation is precisely how the limiting distribution is established by Boca and Zaharescu \cite{boca2007distribution} and by Caglioti and Golse \cite{caglioti2008boltzmann}. These observations were also touched on by Marklof and Str\"{o}mbergsson \cite{marklof2017three} and by Marklof \cite{marklofrandom} concerning the heights of unimodular lattice points. As two dimensional tori are the simplest example of flat surfaces, it is natural to ask if the connections mentioned above extend to other flat surfaces. In this short note, we synthesize these connections and extend them to general flat surfaces in one fell swoop.

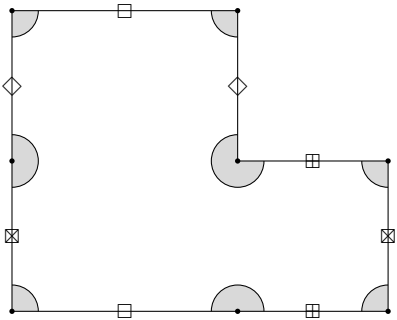
\begin{figure}
\centering
\begin{tikzpicture}
\coordinate (O) at (0, 0);
\coordinate (I) at (3, 0);
\coordinate (A) at (5, 0);
\coordinate (B) at (5, 2);
\coordinate (C) at (3, 2);
\coordinate (D) at (3, 4);
\coordinate (E) at (0, 4);
\coordinate (J) at (0, 2);

\pgfmathsetmacro{\eps}{0.35}
\fill[gray!30] (O) -- (\eps, 0) arc(0:90:\eps) -- cycle;
\draw (\eps, 0) arc(0:90:\eps);
\fill[gray!30] (I) -- (3+\eps, 0) arc(0:180:\eps) -- cycle;
\draw (3+\eps, 0) arc(0:180:\eps);
\fill[gray!30] (A) -- (5, \eps) arc(90:180:\eps) -- cycle;
\draw (5, \eps) arc(90:180:\eps);
\fill[gray!30] (B) -- (5-\eps, 2) arc(180:270:\eps) -- cycle;
\draw (5-\eps, 2) arc(180:270:\eps);
\fill[gray!30] (C) -- (3, 2+\eps) arc(90:360:\eps) -- cycle;
\draw (3, 2+\eps) arc(90:360:\eps);
\fill[gray!30] (D) -- (3-\eps, 4) arc(180:270:\eps) -- cycle;
\draw (3-\eps, 4) arc(180:270:\eps);
\fill[gray!30] (E) -- (0, 4-\eps) arc(270:360:\eps) -- cycle;
\draw (0, 4-\eps) arc(270:360:\eps);
\fill[gray!30] (J) -- (0, 2-\eps) arc(270:450:\eps) -- cycle;
\draw (0, 2-\eps) arc(270:450:\eps);

\draw[->] (O)--(A)--(B)--(C)--(D)--(E)--cycle;
\fill (O) circle[radius=1pt];
\fill (I) circle[radius=1pt];
\fill (A) circle[radius=1pt];
\fill (B) circle[radius=1pt];
\fill (C) circle[radius=1pt];
\fill (D) circle[radius=1pt];
\fill (E) circle[radius=1pt];
\fill (J) circle[radius=1pt];
\node at ($0.5*(O)+0.5*(I)$) {$\scriptstyle\Box$};
\node at ($0.5*(D)+0.5*(E)$) {$\scriptstyle\Box$};
\node at ($0.5*(I)+0.5*(A)$) {$\scriptstyle\boxplus$};
\node at ($0.5*(B)+0.5*(C)$) {$\scriptstyle\boxplus$};
\node at ($0.5*(A)+0.5*(B)$) {$\scriptstyle\boxtimes$};
\node at ($0.5*(J)+0.5*(O)$) {$\scriptstyle\boxtimes$};
\node[rotate=45] at ($0.5*(C)+0.5*(D)$) {$\scriptstyle\Box$};
\node[rotate=45] at ($0.5*(E)+0.5*(J)$) {$\scriptstyle\Box$};

\end{tikzpicture}
\caption{A translation surface $S$ in $\mathcal{H}^1(2)$ arising from a polygon with parallel sides identified. The union of the shaded balls forms the obstacle set $D_\epsilon(S)$.}
\end{figure}

Let $S$ be a unit area flat surface, and let $\Sigma(S) \subset S$ be the set of conical points on $S$. For any positive $\epsilon$ that is smaller than the shortest distance between a pair of distinct conical points of $S$, we write
\begin{equation}
    D_\epsilon(S) := \bigcup_{p \in \Sigma(S)} B^o(p, \epsilon)
\end{equation}
for the open $\epsilon$-neighborhood of $\Sigma(S)$ in $S$, and we consider the free path lengths of the unit speed linear flow on $S \setminus D_\epsilon(S)$. That is, we consider the time it takes a free moving particle on $S \setminus D_\epsilon(S)$ to hit the boundary $\partial D_\epsilon(S)$. Our main theorem on the distribution of free path lengths as $\epsilon \to 0+$ is the following.

\begin{theorem}
Let $\mathcal{H}^1(\alpha)$ be a moduli space of unit area flat surfaces, and let $\mu$ be an ergodic $\SL(2, \RR)$-invariant measure on $\mathcal{H}^1(\alpha)$. Then for $\mu$-a.e. flat surface $S$ in $\mathcal{H}^1(\alpha)$, the limiting average distribution of the free path lengths for the linear flow on $S \setminus D_\epsilon(S)$ exists as $\epsilon \to 0+$ and is equal to the distribution with respect to $\mu$ of zippered rectangle heights over horizontal unit lengths transversals centered at the conical points of the surfaces in $\mathcal{H}^1(\alpha)$.
\end{theorem}

In the remainder of this note, we review basic definitions from the theory of flat surfaces, give precise definitions of free path lengths and their distributions, and prove the main theorem. The zippered rectangle heights distribution mentioned in the theorem is defined in \cref{prop: heights distribution}.

\section{Flat surfaces}

A \emph{flat surface} is a collection of polygons $\{P_j\}_{j=1}^k$ in the plane $\RR^2$ with identifications of parallel sides so that 1. the sides are identified by translations, 2. every side is identified with another side, and 3. the outward pointing normals to identified sides point in opposite directions. If $\sim$ denotes the side identification equivalence relation, then the flat surface is defined as $S = \cup_{j=1}^k P_j/\sim$. Any flat surface is a closed surface with a flat metric outside a finite set of points that are referred to as the \emph{singularities} of the flat surface. Moreover, the conical angle at any point on a flat surface is a positive integer multiple of $2\pi$, and is exactly $2\pi$ except at the singularities.

A natural linear flow $(\phi_S^t)_{t \in \RR}$ can be defined on the unit tangent bundle $T^1(S) = S \times \mathbb{S}^1$ of a flat surface as follows: Given an initial state $(p, \theta) \in T^1(S)$ and a duration $t \in \RR$, move a point particle at $p$ along the geodesic in the direction $\theta$ for time $t$, and denote the final position by $\phi_S^t(p, \theta)$. One has to consider the branching of flat surfaces at singularities when defining linear flows. However, the initial states whose trajectories under $\phi_S^\cdot$ hit a singularity of $S$ form a null subset of the unit tangent bundle $T^1(S)$, and we can ignore those initial states in our measure-theoretic analysis.

The group $\SL(2, \RR)$ acts on planar sets and this action gives rise to a natural action on flat surfaces. Two one-parameter subgroups of $\SL(2, \RR)$ are of particular importance for this work: the rotations
\begin{equation}
    r_\theta := \begin{pmatrix}\cos\theta & -\sin\theta \\ \sin\theta & \cos\theta\end{pmatrix},\ \theta \in \mathbb{S}^1,
\end{equation}
and the scaling matrices
\begin{equation}
    a_t := \begin{pmatrix}e^{t} & 0 \\ 0 & e^{-t}\end{pmatrix},\ t \in \RR.
\end{equation}

In what follows, we denote by $\mathcal{H}^1(\alpha)$, $\alpha = (\alpha_1, \alpha_2, \cdots, \alpha_n) \in \left(\NN \cup \{0\}\right)^n$, the moduli space of all unit area flat surfaces with conical angles $2\pi(\alpha_i + 1)$, $1 \leq i \leq n$. Points which correspond to $\alpha_i = 0$ are customarily referred to as \emph{marked points}.

\section{Free path lengths and their distributions}

Let $S \in \mathcal{H}^1(\alpha)$ be a flat surface, $p \in S$ be a point, $\theta \in \mathbb{S}^1$ be a direction, and $\epsilon > 0$ be an obstacle radius. We are interested in the free path lengths for two types of obstacles:
\begin{itemize}
    \item circular obstacles: denote by $\tau_\epsilon(S, p, \theta)$ the smallest $t \geq 0$ such that $\phi_S^t(p, \theta)$ is in $\overline{D_\epsilon(S)}$, and
    \item segment obstacles: denote by $\tilde{\tau}_\epsilon(S, p, \theta)$ the smallest $t \geq 0$ such that $\phi_S^t(p, \theta)$ can be connected to a point in $\Sigma(S)$ by a line segment perpendicular to $\theta$ of length that does not exceed $\epsilon$.
\end{itemize}

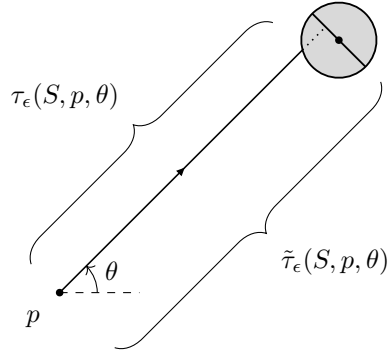
\begin{figure}
\centering
\begin{tikzpicture}[scale=0.5, rotate=45]
\pgfmathsetmacro{\x}{10}

\fill[gray!30] (\x, 0) circle (1);
\draw[line width=0.7] (\x, 0) circle (1);
\fill (\x,0) circle[radius=2.87pt]; 
\draw[line width=0.7] (\x, -1)--(\x, 1);

\draw[line width=0.6] (0, 0.5)--(\x - 1 + 0.13, 0.5) node[currarrow,pos=0.5,rotate=45,scale=0.5] {};
\draw[dotted, line width=0.6] (\x - 1 + 0.13, 0.5)--(\x, 0.5);

\fill (0, 0.5) circle[radius=2.87pt];
\node at (-1,0.5) {$p$};
\draw[dashed] (0,0.5)--(1.5,-1);
\coordinate (O) at (0, 0.5);
\coordinate (I) at (1, 0.5);
\coordinate (P) at (1, -0.5);
\pic [draw, ->, "$\theta$", angle eccentricity=1.5] {angle = P--O--I};

\draw[decorate,decoration={brace,amplitude=10pt,raise=1pt,mirror},yshift=0pt] (0, -1.5) -- (\x, -1.5) node [midway,auto,swap,outer sep=10pt]{$\tilde{\tau}_\epsilon(S, p, \theta)$};
\draw[decorate,decoration={brace,amplitude=10pt,raise=1pt},yshift=0pt] (0, 1.5) -- (\x - 1 + 0.13, 1.5) node [midway,auto,outer sep=10pt]{$\tau_\epsilon(S, p, \theta)$};
\end{tikzpicture}
\caption{The free path lengths $\tau_\epsilon(S, p, \theta)$ and $\tilde{\tau}_\epsilon(S, p, \theta)$ for an initial state $(p, \theta) \in T^1(S)$. The circular obstacle has radius $\epsilon$.}
\end{figure}

Now, we write
\begin{equation}
    F_\epsilon(t; S) := \frac{1}{2\pi} \int_0^{2\pi} \int_S \mathbbm{1}_{(t, \infty)}\left(2\epsilon\tau_\epsilon(S, p, \theta)\right)\,d\vol_S(p)\,d\theta
\end{equation}
for the average distribution of free path lengths for circular obstacles. Similarly, we write
\begin{equation}
    \tilde{f}_\epsilon(t; S, \theta) := \int_S \mathbbm{1}_{(t, \infty)}\left(2\epsilon\tilde{\tau}_\epsilon(S, p, \theta)\right)\,d\vol_S(p)
\end{equation}
for the distribution of free path lengths for segment obstacles perpendicular to $\theta$, and
\begin{equation}
    \tilde{F}_\epsilon(t; S) := \frac{1}{2\pi} \int_0^{2\pi} f_\epsilon(t; S, \theta)\,d\theta
\end{equation}
for the average distribution of free path lengths for segment obstacles in all directions. Following \cite{taha2021threegap}, the distribution $\tilde{f}_\epsilon$ can be interpreted as that of the heights of the zippered rectangle decomposition of $S$ over transversals of lengths $2\epsilon$ centered at $\Sigma(S)$ and perpendicular to $\theta$.

\section{Proof of the main theorem}

The main theorem follows from the following three results.

\begin{lemma}[Approximation]
For any $\epsilon > 0$, $S \in \mathcal{H}^1(\alpha)$, and $t \geq 0$,
\begin{equation}
    F_\epsilon(t; S) = \tilde{F}_\epsilon(t; S) + O_\alpha(\epsilon^2).
\end{equation}
\end{lemma}

\begin{proof}
Consider for any $\theta \in \mathbb{S}^1$ the partition of $S \setminus D_\epsilon(S)$ into the sets
\begin{equation}
R_{i,j}(\theta) := \left\{p \in S \setminus D_\epsilon(S) \middle| \begin{subarray}{c}
\mathbbm{1}_{(t, \infty)}\left(2\epsilon\tau_\epsilon(S, p, \theta)\right)=i\text{, and} \\
\mathbbm{1}_{(t, \infty)}\left(2\epsilon\tilde{\tau}_\epsilon(S, p, \theta)\right)=j
\end{subarray}\right\},
\end{equation}
where the indices $i, j$ run over $\{0, 1\}$. Since
\begin{equation}
    \left|F_\epsilon(t; S) - \tilde{F}_\epsilon(t; S)\right| \leq \frac{1}{2\pi} \int_0^{2\pi}
    \vol_S(R_{0,1}(\theta)) +
    \vol_S(R_{1, 0}(\theta)) +
    \vol_S(D_\epsilon(S)) \,d\theta,
\end{equation}
it suffices to bound $\vol_S(R_{0,1}(\theta))$, $\vol_S(R_{1, 0}(\theta))$, and $\vol_S(D_\epsilon(S))$ to prove the lemma. For any $\theta \in \mathbb{S}^1$ and $p \in S$, we have that $\tau_\epsilon(S, p, \theta) \leq \tilde{\tau}(S, p, \theta)$. It is thus necessarily true that $\vol_S(R_{1, 0}(\theta)) = 0$ for all $\theta \in \mathbb{S}^1$. It is also evident that if $\kappa$ is the sum of the ramification indices of the points in $\Sigma(S)$, then $\vol_S(D_\epsilon(S)) = \kappa \pi \epsilon^2$. Finally, if $p \in R_{0, 1}(\theta)$, then $\phi_S^t(p, \theta) \in D_\epsilon(S)$, and so $\vol_S(R_{0, 1}(\theta)) \leq \kappa \pi \epsilon^2$. The lemma is now proved. 
\end{proof}

\begin{lemma}[Renormalization]
\label{lemm: renormalization}
For any $\epsilon > 0$, $S \in \mathcal{H}^1(\alpha)$, and $t \geq 0$, we have
\begin{equation}
    \label{eq: renormalized heights distribution}
    \tilde{F}_\epsilon(t; S) = \frac{1}{2\pi} \int_0^{2\pi} f_\frac{1}{2}(t; a_{\log\frac{1}{2\epsilon}} r_{- \pi/2 - \theta} S, -\pi/2)\,d\theta.
\end{equation}
\end{lemma}

\begin{proof}
For any $\theta \in \mathbb{S}^1$, the matrix $r_{-\pi/2 - \theta}$ rotates $\theta$ into $-\pi/2$. We thus get that
\begin{equation}
    \tilde{f}_\epsilon(t; S, \theta) = \tilde{f}_\epsilon(t; r_{-\pi/2 - \theta} S, -\pi/2).
\end{equation}
Similarly, the matrix $a_{\log\frac{1}{2\epsilon}}$ scales lengths in the horizontal direction by $\frac{1}{2\epsilon}$, and in the vertical direction by $2\epsilon$. We thus get
\begin{equation}
    \tilde{f}_\epsilon(t; r_{-\pi/2 - \theta} S, -\pi/2) = \tilde{f}_\frac{1}{2}\left(t; a_{\log\frac{1}{2\epsilon}} r_{-\pi/2 - \theta} S, -\pi/2\right).
\end{equation}
This proves the lemma.
\end{proof}

\begin{proposition}[Limiting Distribution]
\label{prop: heights distribution}
Let $\mu$ be an ergodic $\SL(2, \RR)$-invariant probability measure on $\mathcal{H}^1(\alpha)$.
Then for $\mu$-a.e. $S \in \mathcal{H}^1(\alpha)$, and for any $t \in [0, \infty)$
\begin{equation}
    \label{eq: heights distribution limit}
    \lim_{\epsilon \to 0+} \tilde{F}_\epsilon(t; S) = \mathcal{F}_\mu(t),
\end{equation}
where
\begin{equation}
    \mathcal{F}_\mu(t) := \int_{\mathcal{H}^1(\alpha)} f_\frac{1}{2}(t; S, -\pi/2)\,d\mu(S).
\end{equation}
\end{proposition}

\begin{proof}
The proposition follows from \cref{lemm: renormalization} and \cite[theorem 1.5]{eskin2001asymptotic}.
\end{proof}

\bibliographystyle{amsplain}
\bibliography{refs}

\end{document}